\documentclass[11pt,draft]{article}
\usepackage{amsmath,amsthm}
\usepackage{amssymb,latexsym}
\usepackage{enumerate}

\newtheorem{thm}{Theorem}[section]
\newtheorem{cor}[thm]{Corollary}
\newtheorem{lem}[thm]{Lemma}
\newtheorem{prop}[thm]{Proposition}

\usepackage{anysize} 
\marginsize{2cm}{2cm}{2cm}{2cm}

\theoremstyle{definition}

\numberwithin{equation}{section}

\newcommand{\NN}{\mathbb{N}}
\newcommand{\ZZ}{\mathbb{Z}}

\newcommand{\RR}{\mathbb{R}}
\newcommand{\CC}{\mathbb{C}}

\newcommand{\HH}{\mathbb{H}}

\newcommand{\FF}{\mathbb{F}}

\begin{document}

\title{Explicit fundamental solutions of some second order differential operators on Heisenberg groups}

\author{Isolda Cardoso\\
ECEN-FCEIA, Universidad Nacional de Rosario\\
Pellegrini 250, 2000 Rosario, Argentina\\
E-mail: isolda@fceia.unr.edu.ar
\and 
Linda Saal\\
FAMAF, Universidad Nacional de C\'{o}rdoba\\
Ciudad Universitaria, 5000 C\'{o}rdoba, Argentina\\
E-mail: saal@mate.uncor.edu}

\date{}

\maketitle


\renewcommand{\thefootnote}{}



\footnote{\emph{This research was partly supported by research grants form SCyT-UNR and Secyt-UNC, and a fellowship from CONICET.}}

\renewcommand{\thefootnote}{\arabic{footnote}}
\setcounter{footnote}{0}


\begin{abstract}
Let $p,q,n$ be natural numbers such that $p+q=n$. Let $\FF$ be either $\CC$, the complex numbers field, or $\HH$, the quaternionic division algebra. We consider the Heisenberg group $N(p,q,\FF)$ defined as $N(p,q,\FF)=\FF^{n}\times \mathfrak{Im}\FF$, with group law given by $$(v,\zeta)(v',\zeta')=\left( v+v', \zeta+\zeta'-{\frac{1}{2}} \mathfrak{Im} B(v,v') \right),$$ where $B(v,w)=\sum_{j=1}^{p} v_{j}\overline{w_{j}} - \sum_{j=p+1}^{n} v_{j}\overline{w_{j}}$. Let $U(p,q,\FF)$ be the group of $n\times n$ matrices with coefficients in $\FF$ that leave invariant the form $B$. In this work we compute explicit fundamental solutions of some second order differential operators on $N(p,q,\FF)$ which are canonically associated to the action of $U(p,q,\FF)$.
\end{abstract}

\section{Introduction and Preliminaries}
\par Let $p,q,n$ be natural numbers such that $p+q=n$. Let $\FF$ be either $\CC$, the complex numbers field, or $\HH$, the quaternionic division algebra. We consider the Heisenberg group $N(p,q,\FF)$ defined as $N(p,q,\FF)=\FF^{n}\times \mathfrak{Im}\FF$, with group law given by $$(v,\zeta)(v',\zeta')=\left( v+v', \zeta+\zeta'-{\frac{1}{2}} \mathfrak{Im} B(v,v') \right),$$ where $B(v,w)=\sum_{j=1}^{p} v_{j}\overline{w_{j}} - \sum_{j=p+1}^{n} v_{j}\overline{w_{j}}$. The associated Lie algebra is $\eta(p,q,\FF)=\FF^{n}\oplus\mathfrak{Im}(\FF)$, with Lie bracket given by $$[(v,\zeta),(v',\zeta')]=(0,-\mathfrak{Im}B(v,v')).$$ Let $\mathcal{U}(\eta(p,q,\FF))$ be the universal envelopping algebra of $\eta(p,q,\FF)$ which we identify with the algebra of left invariant differential operators. Let $U(p,q,\HH)$ be the group of $n\times n$ matrices with coefficients in $\FF$ that leave invariant the form $B$. Then $U(p,q,\FF)$ acts by automorphism on $\eta(p,q,\FF)$ by $$g\cdot(v,\zeta)=(gv,\zeta).$$ We denote by $\mathcal{U}(\eta(p,q,\FF))^{U(p,q,\FF)}$ the subalgebra of $\mathcal{U}(\eta(p,q,\FF))$ of the left invariant differential operators which commute with this action. It is known that this subalgebra is generated by two operators: $L$ and $U$, and a family of tempered joint eigendistributions is computed explicitly (see for example \cite{[D-M]}, \cite{[G-S1]}, \cite{[V]}).

\par More precisely, if $\FF=\CC$ and $\{X_{1},\dots,X_{n},Y_{1},\dots,Y_{n},U\}$ denotes the standard basis of the Heisenberg Lie algebra with $[X_{i},Y_{j}]=\delta_{ij}U$ and all the other brackets are zero, then $$L=\sum\limits_{j=1}^{p} X_{j}^{2} + Y_{j}^{2} - \sum\limits_{j=p+1}^{n} X_{j}^{2}+Y_{j}^{2}.$$ For $\lambda\in\RR$, $\lambda\neq 0$ and $k\in\ZZ$, $S_{\lambda,k}$ is a $U(p,q)$-invariant tempered distribution satisfying 
$$\begin{array}{rcl}
		LS_{\lambda,k} & = & -|\lambda|(2k+p-q)S_{\lambda,k}, \\
		iUS_{\lambda,k} & = & \lambda S_{\lambda,k}.\end{array}$$
This family provides us an \textit{inversion formula}: for all $f$ in the Schwartz space on the Heisenberg group, we have that
\begin{equation}
\label{inversion1}
f(z,t)=\sum\limits_{k\in\ZZ} \int\limits_{-\infty}^{\infty} f\ast S_{\lambda,k} |\lambda|^{n} d\lambda, \qquad (z,t)\in N(p,q,\CC).
\end{equation}
If $\FF=\HH$ we take $\{X_{1}^{0},X_{1}^{1},X_{1}^{2},X_{1}^{3},\dots,X_{n}^{0},X_{n}^{1},X_{n}^{2},X_{n}^{3},Z_{1},Z_{2},Z_{3}\}$ the canonical basis for the Lie algebra, where $Z_{1},Z_{2},Z_{3}$ generate the center of $\eta(p,q,\HH)$. Here,
\begin{align*}
L= & \sum\limits_{r=1}^{p} \sum\limits_{l=0}^{3}(X_{r}^{l})^{2} - \sum\limits_{r=p+1}^{n} \sum\limits_{l=0}^{3}(X_{r}^{l})^{2},\qquad \mbox{ and} \\
U= & \sum\limits_{l=1}^{3} Z_{l}^{2}.
\end{align*}
There also exists a family of $U(p,q,\HH)-$invariant tempered distributions $\varphi_{w,k}$, $w\in\RR^{3}$ y $k\in\ZZ$, such that each one of them is a joint eigendistribution of $L$ and $U$: \begin{align*} L\varphi_{w,k} &=-|w|(2k+2(p-q))\varphi_{w,k},\\ U\varphi_{w,k} &= -\lambda^{2}\varphi_{w,k};
\end{align*}
in this case this family also provides an inversion formula: for all $f\in\mathcal{S}(N(p,q,\HH))$ we have that \begin{equation}
\label{inversion2}
f(\alpha,z)=\sum\limits_{k\in\ZZ}\int\limits_{\RR^{3}} (f\ast\varphi_{w,k})(\alpha,z)|w|^{2n}dw, \qquad (\alpha,z)\in N(p,q,\HH).
\end{equation}

\par The aim of this work is to explicitly compute a fundamental solution in the classical case for the operator $\mathcal{L}_{\alpha}=L+i\alpha U$, where $\alpha$ is a complex number; and in the quaternionic case for the operator $L$. Recall that a fundamental solution for the differential operator $\mathcal{L}$ is a tempered distribution $\Phi$ such that for all $f$ in the Schwartz class, we have that $\mathcal{L}(f\ast \Phi)=(\mathcal{L}f)\ast\Phi = f\ast \mathcal{L}(\Phi)=f$. So if we define the operator $K$ as $Kf=f\ast\Phi$, then $K\circ\mathcal{L}f=\mathcal{L}\circ Kf=f$.

\par From the inversion formula (\ref{inversion1}) it is natural to propose as a fundamental solution of $\mathcal{L}_{\alpha}$
\begin{equation}
\label{fundsol1}
<\Phi_{\alpha},f>=\sum\limits_{k\in\ZZ}\int\limits_{-\infty}^{\infty} {\frac{1}{{-|\lambda|(2k+p-q-\alpha\mbox{ sgn }\lambda)}}} <S_{\lambda,k},f> |\lambda|^{n} d\lambda,
\end{equation}
for $f\in\mathcal{S}(N(p,q,\CC))$; and from (\ref{inversion2}) we propose as a fundamental solution of $L$
\begin{equation}
\label{fundsol2}
<\Phi,f>=\sum\limits_{k\in\ZZ}\int\limits_{\RR^{3}} {\frac{1}{{-|w|(2k+2(p-q))}}} <\varphi_{w,k},f> |w|^{2n} dw,
\end{equation}
for $f\in \mathcal{S}(N(p,q,\HH))$.

\par We remark that for $q=0$, $\FF=\CC$ we recover the fundamental solution for the operator $\mathcal{L}_{\alpha}$ given in \cite{[F-S]}, and for $q=0$, $\FF=\HH$ we recover Kaplan's fundamental solution for the operator $L$ given in \cite{[K]}. The case $q\neq 0$, $\alpha=0$ was obtained in \cite{[G-S2]}. 

\par The expression of $\Phi_{\alpha}$ is obtained in theorem 2.9, and for the computation we follow the method used in \cite{[G-S2]}. In the quaternionic case, $\Phi$ is given in theorem 3.1, and for its computation we use the Radon transform in order to reduce this case to the classical one.

\hfill

\par To describe both families of eigendistributions $\{S_{\lambda,k}\}$ and $\{\varphi_{w,k}\}$ we need to adapt a result by Tengstrand in \cite{[T]}. We describe the elements for the case $\FF=\CC$, the other one is similar. First of all we take bipolar coordinates on $\CC^{n}$ for $(x_{1},y_{1},\dots,x_{n},y_{n})$ we set $\tau=\sum\limits_{j=1}^{p}(x_{j}^{2}+y_{j}^{2})-\sum\limits_{j=p+1}^{n}(x_{j}^{2}+y_{j}^{2})$, $\rho=\sum\limits_{j=1}^{n}(x_{j}^{2}+y_{j}^{2})$, $u=(x_{1},y_{1},\dots,x_{p},y_{p})$, $v=(x_{p+1},y_{p+1},\dots,x_{n},y_{n})$. Hence $u=\left({{\rho+\tau}\over 2}\right)^{1\over 2}\omega_{u}$, with $\omega_{u}\in S^{2p-1}$ and $v=\left({{\rho-\tau}\over 2} \right)^{1\over 2}\omega_{v}$, with $\omega_{v}\in S^{2q-1}$. It is easy to see by changing variables that \begin{align*}
& \int\limits_{\CC^{n}}f(z)dz=\int\limits_{-\infty}^{\infty}\int\limits_{\rho>|\tau|}\int\limits_{S^{2p-q}\times S^{2q-1}} f\left( \left({{\rho+\tau}\over 2}\right)^{1\over 2}\omega_{u}, \left({{\rho-\tau}\over 2} \right)^{1\over 2}\omega_{v} \right) d\omega_{u} d\omega_{v} \times \\
& \qquad \qquad \times (\rho+\tau)^{p-1}(\rho-\tau)^{q-1}d\rho d\tau.
\end{align*} Then we define for $f\in\mathcal{S}(\RR^{2n})$
$$Mf(\rho,\tau)=\int\limits_{S^{2p-1}\times S^{2q-1}} f\left( \left({{\rho+\tau}\over 2}\right)^{1\over 2}\omega_{u}, \left({{\rho-\tau}\over 2} \right)^{1\over 2}\omega_{v} \right) d\omega_{u} d\omega_{v},$$
and also $$Nf(\tau)=\int\limits_{|\tau|}^{\infty} Mf(\rho,\tau) (\rho+\tau)^{p-1}(\rho-\tau)^{q-1}d\rho.$$ Let us now define the space $\mathcal{H}_{n}$ of the functions $\varphi:\RR\to\CC$ such that $\varphi(\tau)=\varphi_{1}(\tau)+\tau^{n-1}\varphi_{2}(\tau)H(\tau)$, for $\varphi_{1},\varphi_{2}\in\mathcal{S}(\RR)$, where $H$ denotes the Heaviside function. In \cite{[T]} it is proved that $\mathcal{H}_{n}$ with a suitable topology is a Fr\'{e}chet space, and following the same lines we can see that the linear maps $N:\mathcal{S}(\RR^{2n}-\{0\})\to\mathcal{S}(\RR)$ and $N:\mathcal{S}(\RR^{2n})\to\mathcal{H}$ are continuous and surjective. Let us consider now $\mu\in\mathcal{S}'(\RR^{2n})^{U(p,q)}$, then it is easy to see that there exists a unique $T\in\mathcal{S}'(\RR)$ such that $<\mu,f>=<T,Nf>$, for every $f\in\mathcal{S}(\RR^{2n}-\{0\})$. Moreover, if $N':\mathcal{H}'\to\mathcal{S}'(\RR^{2n})$ is the adjoint map, by following again the arguments shown on \cite{[T]} we can see that $N'$ is a homeomorphism. Finally, for a function $f\in\mathcal{S}(N(p,q,\CC))$ we write $Nf(\tau,t)$ for $N(f(.,t))(\tau)$.

\par The distributions $S_{\lambda,k}$ are defined as follows
\begin{equation}
\label{esféricas1}
 S_{\lambda,k}=\sum\limits_{ \{ m\in\NN_{o}^{n}, B(m)=k \} } E_{\lambda}(h_{m},h_{m}),
\end{equation}
 where $B(m)=\sum\limits_{j=1}^{p} m_{j} - \sum\limits_{j=p+1}^{n} m_{j}$, the set of functions $\{ h_{m} \} \subset L^{2}(\RR^{n})$ is the Hermite basis and  $E_{\lambda}(h,h')(z,t)=<\pi_{\lambda}(z,t)h,h'>$ are the matrix entries of the Schr\"{o}dinger representation $\pi_{\lambda}$. Also, $S_{\lambda,k}=e^{-i\lambda t}\otimes F_{\lambda,k}$, where each $F_{\lambda,k}\in\mathcal{S}'(\CC^{n})^{U(p,q)}$ is a tempered distribution defined in terms of the Laguerre polinomials $L_{k}^{m}$ and the map $N$ as follows: for $g\in\mathcal{S}(\CC^{n})$, $\lambda\neq 0$ and $k\in\ZZ$
\begin{align}
	\mbox{if }k\ge 0, & <F_{\lambda,k},g> = < (L_{k-q+n-1}^{0}H)^{(n-1)}, \tau \to {2\over{|\lambda|}} e^{- {\tau \over 2}} Ng\left( {2\over{|\lambda|}} \tau \right)>, \mbox{ and} \label{Flambdak>0} \\
	 \mbox{if }k<0, & <F_{\lambda,k},g>  = < (L_{-k-p+n-1}^{0}H)^{(n-1)}, \tau \to {2\over{|\lambda|}} e^{- {\tau \over 2}} Ng\left(- {2\over{|\lambda|}} \tau\right)>. \label{Flambdak<0} 
\end{align}

\par For the quaternionic case we consider the Schr\"{o}dinger representation $\pi_{w}$ as given in \cite{[R]} (see also \cite{[K-R]}):
\begin{equation}
\label{Schrödinger2}
\pi_{w}(\alpha,z)=\pi_{|w|}\left(\alpha,<z,{w\over{|w|}}>\right),
\end{equation}
where $\pi_{|w|}$ is the Schr\"{o}dinger representation for the classical Heisenberg group $N(2p,2q,\CC)$. We have analogously that the distributions $\varphi_{w,k}$ are defined by 
\begin{equation}
\label{esfericas2} \varphi_{w,k}=\sum\limits_{m\in\NN_{0}^{2n}:B(m)=k} E_{w}(h_{m},h_{m}),
\end{equation}
where $B(m)=\sum\limits_{j=1}^{2p} m_{j} - \sum\limits_{j=2p+1}^{2n} m_{j}$ and $E_{w}(h,h')(\alpha,z)=<\pi_{w}(\alpha,z)h,h'>$ are the matrix entries of the Schr\"{o}dinger representation $\pi_{w}$. Moreover, we have that $\varphi_{w,k}=e^{i<w,z>}\otimes \theta_{w,k}$, where $\theta_{w,k}$ is a tempered distribution such that $\theta_{w,k}=N'T_{|w|,k}$. If we set $\lambda=|w|$, we have that $T_{|w|,k}=F_{\lambda,k}$ where we replace $n,p,q$ by $2n,2p,2q$, respectively, in (\ref{Flambdak>0}) and (\ref{Flambdak<0}). Observe that if we define
\begin{equation}
\label{varphilambda} 
\varphi_{\lambda,k}(\alpha,z) = \int\limits_{S^{2}} e^{i<z,\lambda\xi>} d\xi \theta_{\lambda,k}(\alpha),
\end{equation}
this distributions are $Spin(3)\otimes U(p,q,\HH)-$ invariant.

\section{A fundamental solution for the operator $\mathcal{L}_{\alpha}$}

\par We have that $\Phi_{\alpha}$ defined as in (\ref{fundsol1}) is a well defined tempered distribution, and a fundamental solution for $\mathcal{L}_{\alpha}$. We include the proof since a misprint in Lemma 1 of \cite{[M-R]} is used in the proof Lemma 2.10 of \cite{[B-D-R]}.

\par We will consider $\alpha\in\CC$ such that $2k+p-q\pm \alpha \neq 0$ for all $k\in\ZZ$.
\begin{thm} $\Phi_{\alpha}$ defined as in (\ref{fundsol1}) is a well defined tempered distribution and it is a fundamental solution for the operator $\mathcal{L}_{\alpha}$.
\end{thm}
\begin{proof} From (\ref{fundsol1}) and (\ref{esféricas1}) we can write 
\begin{align*}
	& |<\Phi_{\alpha},f>| \le \sum\limits_{k\in\ZZ} \int\limits_{0}^{\infty} \left( \left| {{<S_{-\lambda,k},f>}\over{(2k+p-q+\alpha)}} \right|+\left| {{<S_{\lambda,k},f>}\over{(2k+p-q-\alpha)}} \right| \right) |\lambda|^{n-1} d\lambda\le \\
	  &  \qquad \le  \sum\limits_{k\in\ZZ} \int\limits_{0}^{\infty}  \sum\limits_{\stackrel{\beta\in\NN_{0}^{n}}{B(\beta)=k}} \left( \left| {{<E_{-\lambda}(h_{\beta},h_{\beta}),f>}\over{(2k+p-q+\alpha)}} \right|  +  \left|{{<E_{\lambda}(h_{\beta},h_{\beta}),f>}\over{(2k+p-q-\alpha)}}\right| \right) |\lambda|^{n-1} d\lambda.
 \end{align*}
From the known facts that $\sum\limits_{k\in\ZZ} \sum\limits_{\stackrel{\beta\in\NN_{0}^{n}}{B(\beta)=k}} p(\beta) = \sum\limits_{k\ge 0} \binom{k+n-1}{n-1} p(k)$, \newline $|<E_{\lambda}(h_{\beta},h_{\beta}),f>|={|<\pi_{\lambda}(f)h_{\beta},h_{\beta}>|}   \le ||f||_{L^{1}(N(p,q,\CC))}$ and  that for $m\in\NN$ $$\pi_{\lambda}(f)h_{\beta}={1\over{(-1)^{m}|\lambda|^{m}(2B(\beta)+p-q+\alpha\mbox{ sgn}(\lambda))^{m}}}\pi_{\lambda}(L^{m}f)h_{\beta}$$ we get that
\begin{align*}
&|<\Phi_{\alpha},f>| \le ||L^{m}f||_{L^{1}(N(p,q,\CC))} \times \\
& \quad \times \sum\limits_{k\ge 0} \int\limits_{0}^{\infty} \binom{k+n-1}{k} \left(  {{|\lambda|^{n-1-m}}\over{|2k+p-q+\alpha|^{m+1}}} + {{|\lambda|^{n-1-m}}\over{|2k+p-q-\alpha|^{m+1}}} \right)    d\lambda.
\end{align*}
Let us consider the first term, the second one is analogous. We split the integral between $|\lambda||2k+p-q+\alpha|\ge 1$ and $0\le |\lambda||2k+p-q+\alpha| \le 1$.Thus,
$$ \sum\limits_{k\ge 0} \binom{k+n-1}{k} \int\limits_{|\lambda||2k+p-q+\alpha|\ge 1}   {1\over{|2k+p-q+\alpha|^{m+1}}}  |\lambda|^{n-1-m} d\lambda $$ is finite if we take $m>n$, and 
$$ \sum\limits_{k\ge 0} \binom{k+n-1}{k} \int\limits_{0\le |\lambda||2k+p-q+\alpha|\le 1}   {1\over{|2k+p-q+\alpha|^{m+1}}}  |\lambda|^{n-1-m} d\lambda $$
is finite for $m<n$.
From the above computations it also follows that $\Phi_{\alpha}$ is a tempered distribution. Next we see that it is a fundamental solution by writing $L=L_{0}+L_{1}$, which in coordinates are 
\begin{align*}
& L_{0}  = {1\over 4} \left( \sum\limits_{j=1}^{p} (x_{j}^{2}+y_{j}^{2}) - \sum\limits_{j=p+1}^{n} (x_{j}^{2}+y_{j}^{2}) \right) {{\partial^{2}}\over{\partial t^{2}}} + \\
&\qquad + \sum\limits_{j=1}^{p} \left( {{\partial^{2}}\over{\partial x_{j}^{2}}}+{{\partial^{2}}\over{\partial y_{j}^{2}}} \right) - \sum\limits_{j=p+1}^{n} \left( {{\partial^{2}}\over{\partial x_{j}^{2}}}+{{\partial^{2}}\over{\partial y_{j}^{2}}} \right), \\
&L_{1}  =  {\partial\over{\partial t}} \sum\limits_{j=1}^{n} \left( x_{j} {\partial\over{\partial y_{j}}} - y_{j} {\partial\over{\partial x_{j}}} \right).
\end{align*}
Then, as $L_{0},L_{1}$ and $T$ commute with left translations and also $L_{0}(g^{\vee})=(L_{0}g)^{\vee}$, $L_{1}(g^{\vee})=-(L_{1}g)^{\vee}$ and $T(g^{\vee})=-(Tg)^{\vee}$ we get that $$(\mathcal{L}f\ast\Phi_{\alpha})(z,t)  =  <\Phi_{\alpha},(L_{(z,t)^{-1}} \mathcal{L}f)^{\vee}>  =  
<\Phi_{\alpha},(L_{0}-i\alpha)(L_{(z,t)^{-1}}f)^{\vee},$$ because $L_{1}\Phi_{\alpha}=0$. Hence, 
\begin{align*}
	& (\mathcal{L}_{\alpha}f\ast\Phi)(z,t)  =   \sum\limits_{k\in\ZZ} \int\limits_{-\infty}^{\infty} {{ <S_{\lambda,k}, (L_{0}-i\alpha T)(L_{(z,t)^{-1}} f)^{\vee}>}\over{-|\lambda|(2k+p-q-\alpha \mbox{ sgn }\lambda)}}  |\lambda|^{n-1} d\lambda  = \\
	& \qquad  =  \sum\limits_{k\in\ZZ} \int\limits_{-\infty}^{\infty} {{<(L_{0}+i\alpha T)S_{\lambda,k}, (L_{(z,t)^{-1}} f)^{\vee}>}\over{-|\lambda|(2k+p-q-\alpha \mbox{ sgn }\lambda)}}  |\lambda|^{n-1} d\lambda  = \\
	& \qquad  =  \sum\limits_{k\in\ZZ} \int\limits_{-\infty}^{\infty} <S_{\lambda,k}, (L_{(z,t)^{-1}} f)^{\vee}> |\lambda|^{n-1} d\lambda  =  f(z,t),
\end{align*}
because of the inversion formula. The other one, $f\ast\mathcal{L}_{\alpha}(f)=f$, is immediate.
\end{proof}
\par Now we proceed with the computation of $\Phi_{\alpha}$. Given that the series (\ref{fundsol1}) defining $\Phi_{\alpha}$ converges absolutely, we can split the sum over $k\in\ZZ$ into the sums for $k\ge q$, for $k\le -p$ and for $-p<k<q$. In the first case we change the summation index writing $k=k'+q$ and in the second as well, writing $k=k'-p$. So we get 
\begin{align*}
	& <\Phi_{\alpha},f> =  (-1)\sum\limits_{k'\ge 0} {1\over{2k'+n-\alpha}} \int\limits_{0}^{\infty}  [<S_{\lambda,k'+q},f>-<S_{\lambda,-k'-p},f>] |\lambda|^{n-1} d\lambda  + \\
	 & \qquad  +  (-1)\sum\limits_{k'\ge 0}  {1\over{2k'+n+\alpha}} \int\limits_{0}^{\infty}  [<S_{-\lambda,k'+q},f>-<S_{-\lambda,-k'-p},f>] |\lambda|^{n-1} d\lambda  + \\
	  & \qquad   +  (-1)\sum\limits_{-p<k<q} \int\limits_{0}^{\infty} \left( {{<S_{-\lambda,k},f>}\over{2k+p-q+\alpha}} + {{<S_{\lambda,k},f>}\over{2k+p-q-\alpha}} \right)   |\lambda|^{n-1} d\lambda. \\
	 	 \end{align*}
\par By Abel's Lemma and the Lebesgue dominated convergence theorem  we can write $\Phi_{\alpha}=\Phi_{1}+\Phi_{2}$ where
\begin{align}
& <\Phi_{1},f> = \lim\limits_{r\to 1^{-}} \lim\limits_{\epsilon\to 0^{+}} (-1)\sum\limits_{k'\ge 0}  {{r^{2k'+n-\alpha}}\over{2k'+n-\alpha}} \int\limits_{0}^{\infty} e^{-\epsilon |\lambda|} \times \\ \nonumber
 & \qquad \qquad  \times \left[<S_{\lambda,k'+q},f>-<S_{\lambda,-k'-p},f>\right] |\lambda|^{n-1} d\lambda  + \\ \nonumber
	& \qquad +  \lim\limits_{r\to 1^{-}} \lim\limits_{\epsilon\to 0^{+}} (-1)\sum\limits_{k'\ge 0}  {{r^{2k'+n+\alpha}}\over{2k'+n+\alpha}} \int\limits_{0}^{\infty} e^{-\epsilon |\lambda|} \times \\ \nonumber
	& \qquad \qquad \times  \left[<S_{-\lambda,k'+q},f>-<S_{-\lambda,-k'-p},f>\right] |\lambda|^{n-1} d\lambda, \label{Phi1}
	\end{align}
\begin{align}
& <\Phi_{2},f>  =  \lim\limits_{\epsilon\to 0^{+}} (-1)\sum\limits_{-p<k<q} \int\limits_{0}^{\infty} e^{-\epsilon |\lambda|} \times \\ \nonumber
& \qquad \qquad \times \left( {{<S_{-\lambda,k},f>}\over{2k+p-q+\alpha}} + {{<S_{\lambda,k},f>}\over{2k+p-q-\alpha}} \right)   |\lambda|^{n-1} d\lambda. \label{Phi2}
\end{align}

\par Using that $S_{\lambda,k}= e^{-i\lambda t}\otimes F_{\lambda,k}$ and the computations from \cite{[G-S2]}, namely (2.6) to (2.9), we get that
\begin{align*}
	&<\Phi_{1},f>  =  \lim\limits_{r\to 1^{-}} \lim\limits_{\epsilon\to 0^{+}} (-1)\sum\limits_{k\ge 0}  {{r^{2k+n-\alpha}}\over{2k+n-\alpha}} \int\limits_{0}^{\infty} e^{-\epsilon |\lambda|} \int\limits_{-\infty}^{\infty} e^{-i\lambda t} \times \\
	 & \qquad  \times  < (L_{k+n-1}^{0} H)^{(n-1)},{2\over{|\lambda|}} e^{- {\tau \over 2}} [Nf( {2\over{|\lambda|}} \tau,t)- Nf( -{2\over{|\lambda|}} \tau,t)]> dt d\lambda + \\
	&  \quad  + \lim\limits_{r\to 1^{-}} \lim\limits_{\epsilon\to 0^{+}} (-1)\sum\limits_{k\ge 0}  {{r^{2k+n+\alpha}}\over{2k+n+\alpha}} \int\limits_{0}^{\infty} e^{-\epsilon |\lambda|} \int\limits_{-\infty}^{\infty} e^{i\lambda t} \times \\
	 & \qquad  \times < (L_{k+n-1}^{0} H)^{(n-1)},{2\over{|\lambda|}} e^{- {\tau \over 2}} [Nf( {2\over{|\lambda|}} \tau,t)- Nf( -{2\over{|\lambda|}} \tau,t)]> dt d\lambda, \\
\end{align*}
and setting 
\begin{equation} \label{bkl} b_{k,l}=\sum\limits_{j=l}^{n-2} \binom{j}{l} \left( {1\over 2} \right)^{2-l} (-1)^{n-j} \binom{k+n-1}{n-j-2},\end{equation} we have that
\begin{align*}
& <\Phi_{1},f>  =  \lim\limits_{r\to 1^{-}} \lim\limits_{\epsilon\to 0^{+}} \sum\limits_{k\ge 0} {r^{2k+n-\alpha}\over{2k+n-\alpha}} \int\limits_{0}^{\infty} e^{-\epsilon |\lambda|} \int\limits_{-\infty}^{\infty} e^{-i\lambda t} \times \\ &  \quad \times \left[ (-1)^{n} \int\limits_{-\infty}^{\infty} L_{k}^{n-1}\left( {{|\lambda|}\over 2} |s| \right) e^{- {{|\lambda|}\over 4} |s|} \mbox{ sgn}(s) Nf(s,t)ds + \right. \\
	 & \left.\qquad
	 - 2 \sum\limits_{\stackrel{l=0}{l \mbox{ \tiny{odd}}}}^{n-2} \left( { 2\over{|\lambda|}} \right)^{l+1} b_{k,l} {{\partial^{l}Nf}\over{\partial\tau^{l}}}(0,t)\right] dt d\lambda + \\ 
	&  + \lim\limits_{r\to 1^{-}} \lim\limits_{\epsilon\to 0^{+}} \sum\limits_{k\ge 0}  {{r^{2k+n+\alpha}}\over{2k+n+\alpha}} \int\limits_{0}^{\infty} e^{-\epsilon |\lambda|} \int\limits_{-\infty}^{\infty} e^{i\lambda t} \times \\ &  \quad \times \left[ (-1)^{n} \int\limits_{-\infty}^{\infty} L_{k}^{n-1}\left( {{|\lambda|}\over 2} |s| \right) e^{- {{|\lambda|}\over 4} |s|} \mbox{ sgn}(s) Nf(s,t)ds + \right. \\
	& \left.\qquad
 - 2 \sum\limits_{\stackrel{l=0}{l \mbox{ \tiny{odd}}}}^{n-2} \left( { 2\over{|\lambda|}} \right)^{l+1} b_{k,l} {{\partial^{l}Nf}\over{\partial\tau^{l}}}(0,t)\right] dt d\lambda.
\end{align*}
Now we define 
\begin{equation}\label{Gf} G_{f}(\tau,t)=Nf(\tau,t)-\sum\limits_{j=0}^{n-2} {{\partial^{j}Nf}\over{\partial \tau^{j}}}(0,t) {{\tau^{j}}\over{j!}},\end{equation}
and then we can split $\Phi_{1}=\Phi_{11}+\Phi_{12}$ where

\begin{align}
&<\Phi_{11},f>=\lim\limits_{r\to 1^{-}}\lim\limits_{\epsilon\to 0^{+}}\sum\limits_{k\ge 0}(-1)^{n}{r^{2k+n-\alpha}\over{2k+n-\alpha}}\int\limits_{0}^{\infty}\int\limits_{-\infty}^{\infty}e^{-\epsilon |\lambda|}e^{-i\lambda t}|\lambda|^{n-1} \times \\ \nonumber
& \qquad \qquad \times \int\limits_{-\infty}^{\infty}L_{k}^{n-1}\left({{|\lambda|}\over 2}|\tau|\right)e^{-{{|\lambda|}\over 4}|\tau|}\mbox{sgn}(\tau)G_{f}(\tau,t)d\tau dtd\lambda+\\ \nonumber
&+\lim\limits_{r\to 1^{-}}\lim\limits_{\epsilon\to 0^{+}}\sum\limits_{k\ge 0}(-1)^{n}{r^{2k+n+\alpha}\over{2k+n+\alpha}}\int\limits_{0}^{\infty}\int\limits_{-\infty}^{\infty}e^{-\epsilon |\lambda|}e^{i\lambda t}|\lambda|^{n-1} \times \\ \nonumber
& \qquad \qquad \times \int\limits_{-\infty}^{\infty}L_{k}^{n-1}\left({{|\lambda|}\over 2}|\tau|\right)e^{- {{|\lambda|}\over 4} |\tau|}\mbox{sgn}(\tau)G_{f}(\tau,t)d\tau dtd\lambda,\label{Phi11} 
\end{align}

\begin{align} 
&	<\Phi_{12},f>  =  \lim\limits_{r\to 1^{-}} \lim\limits_{\epsilon\to 0^{+}} \sum\limits_{k\ge 0} {r^{2k+n-\alpha}\over{2k+n-\alpha}}  \int\limits_{0}^{\infty} \int\limits_{-\infty}^{\infty}  e^{-\epsilon |\lambda|}  e^{-i\lambda t} |\lambda|^{n-1} \times \\ \nonumber
& \qquad \qquad \times 2 \sum\limits_{\stackrel{l=0}{l \mbox{\tiny{ odd}}}}^{n-2} \left( {2\over{|\lambda|}} \right)^{l+1} (a_{k,l}+b_{k,l}) {{\partial^{l}Nf}\over{\partial \tau^{l}}}(0,t) dt d\lambda \; + \\ \nonumber
&	\qquad + \lim\limits_{r\to 1^{-}} \lim\limits_{\epsilon\to 0^{+}} \sum\limits_{k\ge 0} {r^{2k+n+\alpha}\over{2k+n+\alpha}}  \int\limits_{0}^{\infty} \int\limits_{-\infty}^{\infty}  e^{-\epsilon |\lambda|}  e^{i\lambda t} |\lambda|^{n-1} \times \\ \nonumber
& \qquad \qquad \times 2 \sum\limits_{\stackrel{l=0}{l \mbox{\tiny{ odd}}}}^{n-2} \left( {2\over{|\lambda|}} \right)^{l+1} (a_{k,l}+b_{k,l}) {{\partial^{l}Nf}\over{\partial \tau^{l}}}(0,t) dt d\lambda, \\ \label{Phi12}
	\end{align}
	
with \begin{equation} \label{akl} a_{k,l}=(-1)^{n}{1\over{l!}} \int\limits_{0}^{\infty} L_{k}^{n-1}(s)e^{-{s\over 2}} s^{l}ds.\end{equation}
\par We will show that $\Phi_{11}$ is well defined. We have proved that the series (\ref{fundsol1}) defining $\Phi_{\alpha}$ converges and as $\Phi_{2}$ is a finite sum we will obtain that $\Phi_{12}$ is also well defined.

\begin{prop} The following identities hold:
\par $(i)$ \begin{align*}
& \int\limits_{-\infty}^{\infty} e^{-\epsilon |\lambda|} e^{-i\lambda t} L_{k}^{n-1}\left( {|\lambda|\over 2}|\tau| \right) e^{-{|\lambda|\over 4}|\tau|} |\lambda|^{n-1} d\lambda = \\
& \qquad = 4^{n}(n-1)!(-1)^{n} \binom{k+n-1}{k} \left( { {(|\tau|-4\epsilon-4it)^{k}}\over{(|\tau|+4\epsilon+4it)^{k+n}}} \right).\end{align*}
\par $(ii)$ \begin{align*}
&\lim\limits_{\epsilon\to 0^{+}} \int\limits_{\RR^{2}} \left( {{(|\tau|-4it-4\epsilon)^{k}}\over{(|\tau|+4it+4\epsilon)^{k+n}}} \right) \mbox{ sgn}(\tau) G_{f}(\tau,t) d\tau dt = \\
& \quad = \int\limits_{\RR^{2}} {1\over{(|\tau|-4it)^{{n\over 2}-{\alpha\over 2}}}} {1\over{(|\tau|+4it)^{{n\over 2}+{\alpha\over 2}}}} \left( {{|\tau|-4it}\over{\tau^{2}+16t^{2}}} \right)^{2k+n-\alpha} \mbox{ sgn}(\tau) G_{f}(\tau,t) d\tau dt.
\end{align*}
\par $(iii)$ \begin{align*}
&\lim\limits_{\epsilon\to 0^{+}} \int\limits_{\RR^{2}} \left( {{(|\tau|+4it-4\epsilon)^{k}}\over{(|\tau|-4it+4\epsilon)^{k+n}}} \right) \mbox{ sgn}(\tau) G_{f}(\tau,t) d\tau dt = \\
& \quad = \int\limits_{\RR^{2}} {1\over{(|\tau|-4it)^{{n\over 2}-{\alpha\over 2}}}} {1\over{(|\tau|+4it)^{{n\over 2}+{\alpha\over 2}}}} \left( {{|\tau|-4it}\over{\tau^{2}+16t^{2}}} \right)^{2k+n+\alpha} \mbox{ sgn}(\tau) G_{f}(\tau,t) d\tau dt.\end{align*}
\end{prop}
\begin{proof} From (4.9) of \cite{[G-S2]} we know that (i) follows from the generating identity for the Laguerre polynomials: \begin{equation} \label{GI} \sum\limits_{k\ge 0} L_{k}^{n-1}(t)z^{k} = {1\over {(1-z)^{n}}} e^{- { {zt} \over {1-z}}}. \end{equation}
\par From Lemma 2.2 of \cite{[G-S2]}, which states that the function ${{G_{f}(\tau,t)}\over{(\tau^{2}+16t^{2})^{n\over 2}}}$ is integrable in $\RR^{2}$ and the fact that $\left| {1\over {(|\tau|-4it)^{-{\alpha\over 2}}}} \right| \left|{1\over { (|\tau|+4it)^{\alpha\over 2}}}  \right| = 1$, it follows that the function ${1\over {(|\tau|-4it)^{{{n\over  2}-{\alpha\over 2}}}}} {1\over {(|\tau|+4it)^{{n\over 2}+{\alpha \over 2}}}} G_{f}(\tau,t)$
is integrable in $\RR^{2}$. So we get (ii). For (iii) we just change $e^{-i\lambda t}$ by $e^{i\lambda t}$ and argue like for (ii).
\end{proof}

\par Then, by Proposition 2.2, \begin{align*}
&	<\Phi_{11},f>  =  \beta_{n} \lim\limits_{r\to 1^{-}} \sum\limits_{k\ge 0} {r^{2k+n-\alpha}\over{2k+n-\alpha}} \alpha_{k} \times \\
& \qquad \qquad \times \int\limits_{\RR^{2}} \left( { {|\tau|-4it}\over{ \tau^{2}+16t^{2}}} \right)^{2k+n-\alpha} {{\mbox{sgn}(\tau) G_{f}(\tau,t)} \over {(|\tau|-4it)^{{n\over 2}-{\alpha\over 2}}(|\tau|+4it)^{{n\over 2}+{\alpha\over 2}}}}  d\tau dt \,+ \\
& \qquad + \beta_{n} \lim\limits_{r\to 1^{-}} \sum\limits_{k\ge 0} {r^{2k+n+\alpha}\over{2k+n+ \alpha}} \alpha_{k} \times \\
& \qquad \qquad \int\limits_{\RR^{2}} \left( { {|\tau|+4it}\over{ \tau^{2}+16t^{2}}} \right)^{2k+n+\alpha} {{\mbox{sgn}(\tau) G_{f}(\tau,t)} \over {(|\tau|+4it)^{{n\over 2}+{\alpha\over 2}}(|\tau|-4it)^{{n\over 2}-{\alpha\over 2}}}}d\tau dt,
\end{align*} where $\beta_{n}=4^{n}(n-1)!(-1)^{n}$ y $\alpha_{k}=\binom{k+n-1}{k} (-1)^{k}$.
\par To study $<\Phi_{11},f>$ we split each integral over the left and right halfplanes and take polar coordinates $\tau-4it=\rho e^{i\theta}$ to obtain \begin{align*}
&<\Phi_{11},f>  =  \beta_{n} \lim\limits_{r\to 1^{-}} \sum\limits_{k\ge 0} \alpha_{k} {{r^{2k+n-\alpha}}\over{2k+n-\alpha}} \times \\
	& \qquad \qquad \times 	\int\limits_{0}^{\infty} \left[ \int\limits_{-{\pi\over 2}}^{\pi\over 2} e^{i(2k+n-\alpha)\theta} {1\over{4\rho^{n-1}}} e^{i\alpha\theta} \mbox{ sgn }(\cos \theta) G_{f}(\rho\cos\theta,-{\rho\over 4}\sin\theta)d\theta + \right. \\
	& \qquad \qquad  \qquad \left. +  \int\limits_{\pi\over 2}^{{3\over 2}\pi}  {{e^{-i(2k+n-\alpha)\theta}e^{-i\alpha\theta}}\over{(-1)^{n}4\rho^{n-1}} }  \mbox{ sgn }(\cos\theta) G_{f}(\rho\cos\theta,-{\rho\over 4}\sin\theta) d\theta \right] d\rho \,+ \\
	& \qquad + \beta_{n}\lim\limits_{r\to 1^{-}} \sum\limits_{k\ge 0}  \alpha_{k} { {r^{2k+n+\alpha} } \over{2k+n+\alpha}} \times \\
	&  \qquad \qquad \times 	\int\limits_{0}^{\infty} \left[ \int\limits_{-{\pi\over 2}}^{\pi\over 2} e^{-i(2k+n+\alpha)\theta} {1\over{4\rho^{n-1}}} e^{i\alpha\theta} \mbox{ sgn }(\cos \theta) G_{f}(\rho\cos\theta,-{\rho\over 4}\sin\theta)d\theta +  \right. \\
	& \qquad \qquad \qquad \left. + \int\limits_{\pi\over 2}^{{3\over 2}\pi}  {{e^{i(2k+n+\alpha)\theta}e^{-i\alpha\theta}}\over{(-1)^{n}4\rho^{n-1}} }  \mbox{ sgn }(\cos\theta) G_{f}(\rho\cos\theta,-{\rho\over 4}\sin\theta) d\theta \right] d\rho \end{align*} Now we change variable in the second and fourth term according to $\theta \longleftrightarrow -\theta$. Then, in the fourth term we change variables again according to $\theta \longleftrightarrow \theta + 2\pi$. By proposition 2.2 we can interchange the integration order, so we can write \begin{align*}
&<\Phi_{11},f>  =  \beta_{n} \lim\limits_{r\to 1^{-}} \int\limits_{0}^{\infty} \int\limits_{-{\pi\over 2}}^{\pi\over 2} e^{i\alpha\theta} \times \\
& \qquad \qquad \times \left[ \sum\limits_{k\ge 0} \alpha_{k} \left( {{r^{2k+n-\alpha}}\over{2k+n-\alpha}}  e^{i(2k+n-\alpha)\theta} + {{r^{2k+n+\alpha}}\over{2k+n+\alpha}}  e^{-i(2k+n+\alpha)\theta} \right) \right]  \times \\
& \qquad \qquad \times {1\over {\rho^{n-1}}} \mbox{ sgn }(\cos\theta) G_{f}(\rho\cos\theta,-{\rho\over 4}\sin\theta) d\theta d\rho \quad + \\
 & \qquad  + \quad {{(-1)^{n}}\over 4}\beta_{n} \lim\limits_{r\to 1^{-}} \int\limits_{0}^{\infty} \int\limits_{{\pi\over 2}}^{{3\over 2}\pi} e^{i\alpha\theta} \times \\
 & \qquad \qquad \times \left[ \sum\limits_{k\ge 0} \alpha_{k} \left( {{r^{2k+n-\alpha}}\over{2k+n-\alpha}}  e^{i(2k+n-\alpha)\theta} + {{r^{2k+n+\alpha}}\over{2k+n+\alpha}}  e^{-i(2k+n+\alpha)\theta} \right) \right]  \times \\
& \qquad \qquad \times {1\over {\rho^{n-1}}} \mbox{ sgn }(\cos\theta) G_{f}(\rho\cos\theta,{\rho\over 4}\sin\theta) d\theta d\rho.
\end{align*}

\par Let $I$ denote the real interval $\left[-{\pi\over 2},{\pi\over 2}\right]$. Next let us consider the vector space
$$ \mathcal{X} = \left\{ g\in C^{n-2}(I): g^{(j)}\left(\pm {\pi\over 2}\right) = 0,\, 0\le j\le n-2, \, g^{(n-1)}\in L^{\infty}(I) \right\}.$$
We identify each function $g\in\mathcal{X}$ with the function $\widetilde{g}$ on $S^{1}={\RR\over\ZZ}$, defined equal to $0$ outside $\mbox{supp}(g)$ and we make no distinction between $g$ and $\widetilde{g}$. Thus, if $g\in\mathcal{X}$ then $g\in C^{n-2}(S^{1})$ with $g^{(n-1)}\in L^{\infty}(S^{1})$. Observe that if $g\in\mathcal{X}$, then also $e^{i\alpha\theta}g\in\mathcal{X}$. The topology on $\mathcal{X}$ is given by $||g||_{\mathcal{X}}=\mbox{max}_{0\le j\le n-1} ||g^{(j)}||_{\infty}$.

\par For $k\in \ZZ$ we set $\alpha_{k}=\binom{k+n-1}{k} (-1)^{k}$. Now let us define
\begin{equation}
\label{Psir}
\Psi_{r,\alpha}(\theta)  =  \sum\limits_{k\ge 0} \alpha_{k} \left( {{r^{2k+n-\alpha}e^{i(2k+n-\alpha)\theta}}\over{2k+n-\alpha}}   + {{r^{2k+n+\alpha}e^{-i(2k+n+\alpha)\theta}}\over{2k+n+\alpha}}   \right),
\end{equation}
\begin{equation}
\label{Psi}
<\Psi_{\alpha},g>  = < \sum\limits_{k\ge 0} \alpha_{k} \left( {{e^{i(2k+n-\alpha)\theta}}\over{2k+n-\alpha}}   + {{e^{-i(2k+n+\alpha)\theta}}\over{2k+n+\alpha}}   \right) , g>, 
\end{equation}
and let us see that $\Psi_{\alpha}\in\mathcal{X}'$, the dual space of $\mathcal{X}$. Indeed,
\begin{equation}
\label{cotaPsi}
\begin{array}{rcl}
|<\Psi_{\alpha},g>| & \le & |e^{i\alpha\theta}| \sum\limits_{k\ge 0} \binom{k+n-1}{k} ({{|<e^{i(2k+n)\theta},g>|}\over{|2k+n-\alpha|}} + |{{<e^{-i(2k+n)\theta},g>|}\over{|2k+n+\alpha|}}).
\end{array}
\end{equation}
If $\widehat{g}(n)=<g,e^{i n\theta}>$ denotes the $n-$th Fourier coefficient of $g$, then we have that
\begin{align*}
& |<\Psi_{\alpha},g>| \le  c \sum\limits_{k\ge 0} {{k^{n-1}}\over{|2k+n|^{n-1}}} \left( {{|\widehat{g^{(n-1)}}(2k+n)|}\over{|2k+n-\alpha|}} + {{|\widehat{g^{(n-1)}}(-2k-n)|}\over{|2k+n+\alpha|}} \right)  \le \\
& \qquad \qquad \le  c \sum\limits_{k\ge 0} {1\over k} {|\widehat{g^{(n-1)}}(2k+n)|} + {1\over k} {|\widehat{g^{(n-1)}}(-2k-n)|} \le \\
& \qquad \qquad \le c \left( \sum\limits_{k\ge 0} {1\over{k^{2}}} \right)^{1\over 2} ||\widehat{g^{(n-1)}}||_{L^{2}}, 
\end{align*}
where we used the Cauchy-Schwarz inequality. Observe that the constants $c$ are not the same on each expression.
\par By Abel's Lemma, $\lim_{r\to 1^{-}} \Psi_{r,\alpha}=\Psi_{\alpha}$ in $\mathcal{X}'$ that is, with respect to the weak convergence topology. Similarly, if $J$ denotes the real interval $\left[{\pi\over 2},{3\over 2}\pi \right]$, we define the space,
$$\mathcal{Y}=\left\{ g\in C^{n-2}(J): g^{(j)}\left({\pi\over 2}\right) = g^{(j)}\left({3\over 2}\pi\right) = 0,\, 0\le j\le n-2, \, g^{(n-1)}\in L^{\infty}(J) \right\},$$ 
and we obtain that $\Psi_{\alpha}$ is well defined in $\mathcal{Y}'$ and $\lim\limits_{r\to 1^{-}} \Psi_{r,\alpha} = \Psi_{\alpha}$ in $\mathcal{Y}'$.
\par Our aim now is to compute $\Psi_{\alpha}$.
\par From Proposition 3.7 of \cite{[G-S2]} we know that if $\Theta\in\mathcal{D}'(S^{1})$ is defined by
\begin{equation}
\Theta(\theta)= i\sum\limits_{k\ge 0} \binom{k+n-1}{k} (-1)^{k} e^{i(2k+n)\theta},
\label{Theta}
\end{equation}
then for even $n$ we have that
\begin{equation}
\mathfrak{Re}\Theta(\theta)={d\over {d\theta}} Q_{n-2}\left({d\over{d\theta}}\right) (\delta_{\pi\over 2}+\delta_{-{\pi\over 2}})=\sum\limits_{j=0}^{n-2} c_{j}\left(\delta^{(j+1)}_{\pi\over 2}+\delta^{(j+1)}_{-{\pi\over 2}}\right),
\label{ThetaPAR}
\end{equation}
where $Q_{n-2}$ is a polynomial of degree $n-2$; and for odd $n$ we have that
\begin{equation}
\mathfrak{Re}\Theta(\theta)=d_{0}{d\over {d\theta}}\widetilde{H} + {d\over {d\theta}} Q_{n-2}\left({d\over{d\theta}}\right) (\delta_{\pi\over 2}-\delta_{-{\pi\over 2}}) = d_{0}(\delta_{-{\pi\over2}} - \delta_{\pi\over 2}) + \sum\limits_{j=0}^{n-2} c_{j} (\delta^{(j+1)}_{\pi\over 2}-\delta^{(j+1)}_{-{\pi\over 2}}),
\label{ThetaIMPAR}
\end{equation}
where $Q_{n-2}$ is a polynomial of degree $n-2$, and $\widetilde{H}(\theta)=H(\cos\theta)$.
\par Let us recall the generating identity for the Laguerre polynomials (\ref{GI}), and take $t=0$ and $z=-r^{2}e^{2i\theta}$. We get
\begin{equation}
\sum\limits_{k\ge 0} \binom{k+n-1}{k} (-1)^{k} r^{2k+n} e^{i(2k+n)\theta} = \left( { {r e^{i\theta}}\over {1+r^{2}e^{2i\theta}}} \right)^{n}.
\label{CU}
\end{equation}
\par We also need a couple of results:
\begin{lem} For a fixed $r>1$ the functions $\alpha\to \Psi_{r,\alpha}(0)$ and $\alpha\to \lim\limits_{r\to 1^{-}} \Psi_{r,\alpha}(0)$ are analytic on $\Omega=\CC\backslash F$, where $F=\{2k+n:k\in\ZZ\}$.
\end{lem}
\begin{proof} Let $K$ be a compact set, $K\subset\Omega$. It is easy to see that for fixed $r$ the series (\ref{Psir}) converges uniformly, since $$\left| \Psi_{r,\alpha}(0)\right| \le \max\limits_{\alpha\in K} |r^{\mathfrak{ \alpha}}| \left( {r\over{1+r^{2}}} \right)^{n} d(K,F).$$
Also, for $\alpha\in\Omega$ there exists $\lim\limits_{r\to 1^{-}} \Psi_{r,\alpha}(0)$. Indeed, if $0\le r_{1}<r<r_{2}<1$, from the Mean Value Theorem we have that for some $\xi\in(r_{1},r_{2})$, \begin{align*}
&\Psi_{r_{1},\alpha}(0)-\Psi_{r_{2},\alpha}(0)= {d\over dr}\Psi_{\xi,\alpha}(0)(r_{2}-r_{1}) = (\xi^{-\alpha-1}+\xi^{\alpha-1})\sum\limits_{k\ge 0} \alpha_{k} \xi^{2k+n} (r_{2}-r_{1})= \\
& \qquad = (\xi^{-\alpha-1}+\xi^{\alpha-1}) \left( { \xi\over{1+\xi^{2}}} \right)^{n}(r_{2}-r_{1}),\end{align*} where the last equality holds from (\ref{CU}). Hence $$|\Psi_{r_{1},\alpha}(0)-\Psi_{r_{2},\alpha}(0)|\le c(\xi) |r_{2}-r_{1}|,$$ where $c(\xi)$ is a constant which depends on $\xi$. Moreover, for $\alpha\in K$ and $\xi\in\left[ {1\over 2}, 1\right]$, $\xi^{n-\alpha-1}+\xi^{n+\alpha-1}$ is bounded in $K\times \left[ {1\over 2}, 1\right]$, so the convergence is uniform, hence $\alpha\to\lim\limits_{r\to 1^{-}} \Psi_{r,\alpha}(0)$ is an analytic function.
\end{proof}

\begin{lem} Let $0<\delta<{\pi\over 4}$. For $0<r<1$ and $0\le |\theta|<\delta$ we have that
$$|\Psi_{r,\alpha}(\theta)-\Psi_{r,\alpha}(0)|\le \left(\max\limits_{0\le|\theta|<\delta} e^{|\mathfrak{Im}\alpha||\theta|} \right) \left(a|r^{-\alpha}-r^{\alpha}| + b |r^{\alpha}|(1-r)\right)|\theta|,$$ with $a,b$ positive constants. Also for $0\le|\theta-\pi|<\delta<{\pi\over 4}$, $$|\Psi_{r,\alpha}(\theta)-\Psi_{r,\alpha}(\pi)|\le \left(\max\limits_{0\le|\theta-\pi|<\delta} e^{|\mathfrak{Im}\alpha||\theta|} \right) \left(a|r^{-\alpha}-r^{\alpha}| + b |r^{\alpha}|(1-r)\right)|\theta-\pi|,$$ with $a,b$ positive constants.
\end{lem}
\begin{proof} We will estimate $|\Psi_{r,\alpha}(\theta)-\Psi_{r,\alpha}(0)|$ for $0<|\theta|<\delta<{\pi\over 4}$, the other case is similar. We have
\begin{align*}
& {d\over{d\theta}}\Psi_{r,\alpha}(\theta)  =  ie^{-i\alpha\theta} \sum\limits_{k\ge 0} \alpha_{k} r^{2k+n} \left( (r^{-\alpha}-r^{\alpha})e^{i(2k+n)\theta} + (e^{i(2k+n)\theta}-e^{-i(2k+n)\theta})r^{\alpha}\right)  =\\
& \qquad = ie^{-i\alpha\theta} \left( (r^{-\alpha}-r^{\alpha}) \left( {{re^{i\theta}}\over{1+r^{2}e^{2i\theta}}} \right)^{n} + 2i r^{\alpha} \mathfrak{Im} \left( { {re^{i\theta}} \over  {1+r^{2}e^{i2\theta}} } \right)^{n} \right), \end{align*} 
because of (\ref{CU}). We have that
\begin{equation}
\left| {d\over d\theta}\Psi_{r,\alpha}(\theta) \right|  \le e^{|\mathfrak{Im}\alpha||\theta|} \left( |r^{-\alpha}-r^{\alpha}| \left| \left( { {re^{i\theta}} \over  {1+r^{2}e^{i2\theta}} } \right)^{n} \right| + 2 |r^{\alpha}| \left| \mathfrak{Im} \left( { {re^{i\theta}} \over  {1+r^{2}e^{i2\theta}} } \right)^{n} \right| \right). 
\label{acotdpsir}
\end{equation}
\par From Proposition 3.1 of \cite{[G-S2]} we know that $\left| \mathfrak{Im} \left( { {re^{i\theta}} \over  {1+r^{2}e^{i2\theta}} } \right)^{n} \right| \longrightarrow   0$ as $r\to 1^{-}$, uniformly for $|\theta|<{\pi\over 4}$, $|\theta - \pi|<{\pi\over 4}$. Also, $\left| \left( { {re^{i\theta}} \over  {1+r^{2}e^{i2\theta}} } \right)^{n} \right| \le c$, for a constant $c$. Then $ \left| {d\over d\theta} \Psi_{r,\alpha} (\theta) \right| \longrightarrow 0$ uniformly on $|\theta|<{\pi\over 4}$ as $r\to 1^{-}$, and we get the desired inequality by applying the Mean Value Theorem around $0$.
\end{proof}

\par Now we can state the following
\begin{prop} For $f\in\mathcal{X}$ we have that $<\Psi_{\alpha},f>=C_{\alpha}<1,f>$, where $C_{\alpha}={{\Gamma\left({{n+\alpha}\over 2}\right) \Gamma\left({{n-\alpha}\over 2}\right)}\over{(n-1)!}}$; and for $f\in\mathcal{Y}$ we have that $<\Psi_{\alpha},f>=\widetilde{C_{\alpha}}<1,f>$, where $\widetilde{C_{\alpha}}=(-1)^{n} e^{-i\alpha\pi} C_{\alpha}$.
\end{prop}
\begin{proof} First we consider $f\in\mathcal{X}$ such that $\int\limits_{-{\pi\over 2}}^{\pi\over 2} f(t) dt =0$ and we define $F(\theta)=\int\limits_{-{\pi\over 2}}^{\theta} f(t) dt$. It is easy to see that $F\in\mathcal{X}$ and $F'=f$. Because of the integration by parts formula we have that
\begin{align*}
<\Psi_{\alpha},f> =&  <\Psi_{\alpha},F'>  = \int\limits_{-{\pi\over 2}}^{\pi\over 2} \sum\limits_{k\ge 0} \alpha_{k} \left( {{e^{i(2k+n-\alpha)\theta}\over{2k+n-\alpha}}} + {{e^{-i(2k+n+\alpha)\theta}\over{2k+n+\alpha}}} \right) F'(\theta)d\theta = \\
 =& -<\Theta,e^{-i\alpha\theta}F>-<\overline{\Theta},e^{-i\alpha\theta}F>,
\end{align*}
where $\overline{\Theta}=\sum\limits_{k\ge 0} \binom{k+n-1}{k} (-1)^{k}e^{-i(2k+n)\theta}$. So if $n$ is even, from (\ref{ThetaPAR}) we get that $$
	<\Psi_{\alpha},f> = - \sum\limits_{j=0}^{n-2} c_{j} <\delta^{(j+1)}_{\pi\over 2}+\delta^{(j+1)}_{-\pi\over 2},e^{-i\alpha\theta}F>  - \sum\limits_{j=0}^{n-2} \overline{c_{j}}  <\overline{\delta^{(j+1)}_{\pi\over 2}}+\overline{\delta^{(j+1)}_{-\pi\over 2}},e^{-i\alpha\theta}F>,$$
	and because $<\delta_{\pm {\pi\over 2}}^{(j+1)}, e^{-i\alpha\theta}F>=0$ we conclude that $<\Psi_{\alpha},f> = 0$. If $n$ is odd we use (\ref{ThetaIMPAR}) to conclude that $<\Psi_{\alpha},f> = 0$.
\par For the general case of any $f\in\mathcal{X}$ we consider $h\in\mathcal{X}$ such that $\int_{-{\pi\over 2}}^{\pi\over 2} h(t)dt=1$ and define $g(\theta)=f(\theta)-\left( \int_{-{\pi\over 2}}^{\pi\over 2} f(t) dt \right) h(\theta)$. So we can apply the above result to $g$ and get that $<\Psi_{\alpha},g>=0$. Then $<\Psi_{\alpha},f>=<\Psi_{\alpha},g>+<\Psi_{\alpha},h><1,f>=<\Psi_{\alpha},h><1,f>$. Let $C_{\alpha}=<\Psi_{\alpha},h>$.

\par In order to compute $C_{\alpha}$, consider $g\in\mathcal{X}$ such that $\mbox{ supp}(g) \subset (-{\pi\over 4},{\pi\over 4})$, $\int\limits_{-\pi\over 4}^{\pi\over 4} g(t) dt =1$ and $g\ge 0$ we have that $$<e^{i\alpha\theta}\Psi_{\alpha},g> = C_{\alpha}  \int\limits_{-\pi\over 2}^{\pi\over 2} e^{i\alpha\theta} g(\theta) d\theta,$$ and also that $$<e^{i\alpha\theta}\Psi_{\alpha},g> = \lim\limits_{r\to 1^{-}} \left( \int\limits_{-{\pi\over 2}}^{\pi\over 2} (\Psi_{r,\alpha}(\theta)-\Psi_{r,\alpha}(0))e^{i\alpha\theta}g(\theta)d\theta + \Psi_{r,\alpha}(0) \int\limits_{-{\pi\over 2}}^{\pi\over 2} e^{i\alpha\theta}g(\theta)d\theta \right).$$ From lemmas 1 and 2 we get that $$C_{\alpha}=\lim\limits_{r\to 1^{-}} \Psi_{r,\alpha}(0)$$ and also that $C_{\alpha}$ is an analytic function of $\alpha$. Given that $\Psi_{0,\alpha}(0)=0$ we can write $$C_{\alpha}=\lim\limits_{r\to 1^{-}} \Psi_{r,\alpha}(0) = \Psi_{1,\alpha}(0)-\Psi_{0,\alpha}(0) = \int\limits_{0}^{1} w'_{\alpha}(s) ds,$$ where 
$$w_{\alpha}(r)  =  \Psi_{r,\alpha}(0)  = r^{-\alpha} \sum\limits_{k\ge 0} \alpha_{k} {{r^{2k+n}\over{2k+n-\alpha}}} + r^{\alpha} \sum\limits_{k\ge 0} \alpha_{k} {{r^{2k+n}\over{2k+n+\alpha}}}.$$
Applying (\ref{GI}) with $\theta=0$ we obtain $w_{\alpha}'(r)= (r^{-\alpha-1}+r^{\alpha-1}) \sum\limits_{k\ge 0} \alpha_{k}r^{2k+n} = (r^{-\alpha-1}+r^{\alpha-1}) \left( {r\over {1+r^{2}}} \right)^{n}$, and we can solve the integral for $\mathfrak{Re}(n+\alpha)>0$, $\mathfrak{Re}(n-\alpha)>0$, getting
\begin{equation}
\label{C0}
C_{\alpha}=B\left( {{n+\alpha}\over 2} , {{n-\alpha}\over 2} \right)={{\Gamma\left({{n+\alpha}\over 2}\right) \Gamma\left({{n-\alpha}\over 2}\right)}\over{(n-1)!}},
\end{equation}
where $B$ is the \textit{Beta function} and $\Gamma$ is the \textit{Gamma function}. By lemma 2.3, (\ref{C0}) holds for the other range of $\alpha$, by analytic continuation. In a completely analogous way we get that $\widetilde{C_{\alpha}}=(-1)^{n} e^{-i\alpha\pi} C_{\alpha}$.
\end{proof}

\par Let us now define
\begin{equation}
\label{K1f}
K_{1f}(\rho,\theta)={1\over \rho^{n-1}} \mbox{sgn}(\cos\theta) G_{f}(\rho\cos\theta, -{\rho\over 4}\sin\theta),
\end{equation}
for $\theta\in\left[-{\pi\over 2},{\pi\over 2}\right]$, $0<\rho<\infty$, where $G_{f}$ is the function defined in (\ref{Gf}); and
\begin{equation}
\label{K2f}
K_{2f}(\rho,\theta)={1\over \rho^{n-1}} \mbox{sgn}(\cos\theta) G_{f}(\rho\cos\theta, {\rho\over 4}\sin\theta),
\end{equation}
for $\theta\in\left[{\pi\over 2},{3\over 2}\pi\right]$, $0<\rho<\infty$.
\par It is easy to check that $K_{1f}(\rho,.)\in\mathcal{X}$. Recall that we changed variables according to $\tau-4it=\rho e^{i\theta}$. Since $Nf\in\mathcal{H}_{n}$, there exists a positive constant $c$ such that 
$$\sup\limits_{\tau\neq 0, t\in\RR} |(\tau^{2}+16t^{2})Nf(\tau,t)|\le c,$$ that is $$\left| Nf(\rho\cos\theta,-{\rho\over 4}\sin{\theta}) \right|\le {c\over\rho^{2}}.$$ Also since $Nf(0,.)\in\mathcal{S}(\RR)$, there exists a positive constant $c_{N}$ such that for $t\in\RR$, $$\left| t^{N} \sum\limits_{j=0}^{n-2} {{\partial^{j}}\over{\partial\tau^{j}}} Nf(0,t) {{\tau^{j}}\over{j!}} \right| \le c_{N} |\tau|^{n-2}.$$ Thus, for $N\in\NN$ there exists $c_{N}$ such that 
\begin{equation}
\label{cotaK1f}
|K_{1f}(\rho,\theta)|\le {a\over{\rho^{n+1}}} + {b\over{\rho^{N+1}}} {{|\cos\theta|^{n-2}}\over{|\sin\theta|^{N}}}.
\end{equation}
Analogous observations are also true for $K_{2f}$. 
\begin{prop} Let $C_{\alpha}$ and $\widetilde{C_{\alpha}}$ be the constans obtained in (\ref{C0}). Let $K_{1f}$ and $K_{2f}$ be as defined in (\ref{K1f}) and (\ref{K2f}) and $\alpha_{k}=\binom{k+n-1}{k} (-1)^{k}$. Then
\begin{align*}
& \lim\limits_{r\to 1^{-}} \int\limits_{0}^{\infty}  \int\limits_{-{\pi\over 2}}^{\pi\over 2} e^{i\alpha\theta} \sum\limits_{k\ge 0} \alpha_{k} \left( {{ r^{2k+n-\alpha}e^{i(2k+n-\alpha)\theta}} \over {2k+n-\alpha}}  + {{ r^{2k+n+\alpha}e^{-i(2k+n+\alpha)\theta}}  \over {2k+n+\alpha}} 
 \right) K_{1f}(\rho,\theta)d\theta d\rho =  \\
 & \qquad = 4^{n-1}(n-1)!C_{\alpha} \int\limits_{\RR} \int\limits_{\tau>0} {1\over {(\tau-4it)^{{n-\alpha}\over 2}}} {1\over {(\tau+4it)^{{n+\alpha}\over 2}}} \mbox{sgn}(\tau)G_{f}(\tau,t)d\tau dt,
  \end{align*}
and
\begin{align*}
& \lim\limits_{r\to 1^{-}} \int\limits_{0}^{\infty}  \int\limits_{\pi\over 2}^{{3\over 2}\pi} e^{i\alpha\theta} \sum\limits_{k\ge 0} \alpha_{k} \left( {{ r^{2k+n-\alpha}e^{i(2k+n-\alpha)\theta }} \over {2k+n-\alpha}}  + {{ r^{2k+n+\alpha}e^{-i(2k+n+\alpha)\theta}}  \over {2k+n+\alpha}} 
 \right) K_{2f}(\rho,\theta)d\theta d\rho =  \\ 
 & \qquad = 4^{n-1}(n-1)!\widetilde{C_{\alpha}} \int\limits_{\RR} \int\limits_{\tau<0} {1\over {(\tau-4it)^{{n-\alpha}\over 2}}} {1\over {(\tau+4it)^{{n+\alpha}\over 2}}} \mbox{sgn}(\tau)G_{f}(\tau,t)d\tau dt.
  \end{align*}
\end{prop}
\begin{proof} The proof follows the same lines that Proposition 4.2 of \cite{[G-S2]}. We only sketch it for completeness.
\par Taking polar coordinates $\tau-4it=\rho e^{i\theta}$ we only need to show that
\begin{equation}
\label{integral}
\lim\limits_{r\to 1^{-}} \int\limits_{0}^{\infty} <\Psi_{r,\alpha},e^{i\alpha\theta}K_{1f}(\rho,\theta)> d\rho = \int\limits_{0}^{\infty} <C_{\alpha},e^{i\alpha\theta}K_{1f}(\rho,\theta)> d\rho.
\end{equation}
\par In order to do this we split the integral for $0<\rho<1$ and $1<\rho<\infty$.
\par We consider first the case $1<\rho<\infty$. For $|\theta|\le \delta <{\pi\over 4}$, set $ I= \int\limits_{1}^{\infty} \int\limits_{|\theta|<\delta} e^{i\alpha\theta}(\Psi_{r,\alpha}(\theta)-\Psi_{r,\alpha}(0))K_{1f}(\rho,\theta) d\theta d\rho$ and $II= \int\limits_{1}^{\infty} \int\limits_{|\theta|<\delta} e^{i\alpha\theta}(\Psi_{r,\alpha}(0)-C_{\alpha})K_{1f}(\rho,\theta) d\theta d\rho$. We bound $I$ close to $0$ by applying Lemma 2.4 and taking $N=1$ in (\ref{cotaK1f}). For $II$ we just take $N={1\over 2}$ in (\ref{cotaK1f}). To analize the case $\delta\le|\theta|\le {\pi\over 2}$, we observe that the function $K_{1f}^{\ast}(\theta)=\int\limits_{1}^{\infty} K_{1f}(\rho,\theta) d\rho$ defined for $\theta\in[-{\pi\over2}, -\delta]\cup[\delta,{\pi\over 2}]$ can be extended to an element of $\mathcal{X}$ that we still denote by $K_{1f}^{\ast}$. Then
\begin{align*}
&\int\limits_{1}^{\infty} \int\limits_{\delta<|\theta|<{\pi\over 2}} e^{i\alpha\theta} (\Psi_{r,\alpha}(\theta)-C_{\alpha}) K_{1f}(\rho,\theta) d\theta d\rho = \\
& \qquad = \int\limits_{-{\pi\over 2}}^{\pi\over 2} e^{i\alpha\theta} (\Psi_{r,\alpha}(\theta)-C_{\alpha}) K_{1f}^{*}(\theta) d\theta - \int\limits_{|\theta|<\delta} e^{i\alpha\theta} (\Psi_{r,\alpha}(\theta)-C_{\alpha}) K_{1f}^{*}(\theta) d\theta.\end{align*}
The first term converges to zero as $r\to 1^{-}$ since $\Psi_{r,\alpha}\to C_{\alpha}$ as $r\to 1^{-}$ in $\mathcal{X}'$. For the second term we argue as above.
\par Finally, for the case $0<\rho<1$ we apply the same arguments to the function $K_{1f}^{\ast\ast}(\theta)=\int\limits_{0}^{1} K_{1f}(\rho,\theta) d\rho.$
\end{proof}

\begin{cor} $<\Phi_{11},f>$ is well defined for $f\in\mathcal{S}(\mathbb{H}_{n})$, and \begin{align*}
&<\Phi_{11},f>= 4^{n-1}(n-1)!C_{\alpha} \int\limits_{\RR} \int\limits_{\tau>0} {1\over {(\tau-4it)^{{n-\alpha}\over 2}}} {1\over {(\tau+4it)^{{n+\alpha}\over 2}}} \mbox{sgn}(\tau)G_{f}(\tau,t)d\tau dt + \\
& \qquad  +  4^{n-1}(n-1)!\widetilde{C_{\alpha}} \int\limits_{\RR} \int\limits_{\tau<0} {1\over {(\tau-4it)^{{n-\alpha}\over 2}}} {1\over {(\tau+4it)^{{n+\alpha}\over 2}}} \mbox{sgn}(\tau)G_{f}(\tau,t)d\tau dt.\end{align*}
\end{cor}

\par From the corollary we also get that $<\Phi_{12},f>$ is well defined. In order to explicitly compute it we define for $0\le l \le n-2$, $\epsilon >0$ and $f\in\mathcal{S}(\mathbb{H}_{n})$
\begin{equation}
\label{d-elf}
d^{-}_{\epsilon,l,f}=\int\limits_{0}^{\infty} \int\limits_{-\infty}^{\infty} e^{-\epsilon |\lambda|} e^{-i\lambda t} |\lambda|^{n-l-2}{ {\partial^{l}}\over{\partial \tau^{l}} } Nf(0,t) dt d\lambda, \quad \mbox{ and }
\end{equation}
\begin{equation}
\label{d+elf}
d^{+}_{\epsilon,l,f}=\int\limits_{0}^{\infty} \int\limits_{-\infty}^{\infty} e^{-\epsilon |\lambda|} e^{i\lambda t} |\lambda|^{n-l-2}{ {\partial^{l}}\over{\partial \tau^{l}} } Nf(0,t) dt d\lambda.
\end{equation}
Then we can write (\ref{Phi12}) as
$$<\Phi_{12},f> = \lim\limits_{r\to 1^{-}} \lim\limits_{\epsilon \to 0^{+}} \sum\limits_{k\ge 0} \sum\limits_{\stackrel{l=0}{l \mbox{ odd}}}^{n-2} 2^{l+2} (a_{kl}+b_{kl})  \left[ { {r^{2k+n-\alpha}} \over {2k+n-\alpha}} d^{-}_{\epsilon,l,f} + 
 { {r^{2k+n+\alpha}} \over {2k+n+\alpha}} d^{+}_{\epsilon,l,f}\right].$$
\par From Lemma 4.4 in \cite{[G-S2]} we deduce that
$$a_{kl}+b_{kl}= (-1)^{k} \sum\limits_{j=1}^{l+1} {1\over{2^{n-l-j-1}}} \binom{n-j-1}{l-j+1} \binom{j+k-1}{k}.$$
Also we have the following
\begin{lem} If $0\le l \le n-2$, $\epsilon >0$ and $f\in\mathcal{S}(\mathbb{H}_{n})$, then
$$\lim\limits_{\epsilon\to 0^{+}} d^{-}_{\epsilon,l,f} = {1\over {i^{n-l-2}}} < {\pi\over 2}\delta-i (vp)\left({1\over\lambda}\right), {{\partial^{n-2}}\over{\partial\lambda^{n-l-2}\partial\tau^{l}}} Nf(0,.)>
\quad \mbox{ and }$$
$$\lim\limits_{\epsilon\to 0^{+}} d^{+}_{\epsilon,l,f} = i^{n-l-2} < {\pi\over 2}\delta+i (vp)\left({1\over\lambda}\right), {{\partial^{n-2}}\over{\partial\lambda^{n-l-2}\partial\tau^{l}}} Nf(0,.)>.$$
\end{lem}
\begin{proof} Let us consider $g(\lambda)=e^{-\epsilon|\lambda|} |\lambda|^{n-l-2}$ and $h(t)={{\partial^{l}}\over{\partial \tau^{l}}} Nf(0,t)$, and observe that $\int_{-\infty}^{\infty} e^{-i\lambda t} h(t) dt = \hat{h}(\lambda)$. Then just by using properties of the Fourier transform we get that
$$d^{-}_{\epsilon,l,f}=\int\limits_{0}^{\infty}\int\limits_{-\infty}^{\infty} g(\lambda) e^{-i\lambda t} h(t) dt d\lambda = \int\limits_{0}^{\infty} g(\lambda) \hat{h}(\lambda) d\lambda =  {1\over {i^{n-l-2}}}\int\limits_{-\infty}^{\infty} {1\over {\epsilon+i\lambda}} h^{(n-l-2)}(\lambda) d\lambda
.$$
\par For each $\epsilon >0$, $1\over{\epsilon+i\lambda}$ is a distribution such that there exists $\lim\limits_{\epsilon\to 0^{+}} {1\over{\epsilon+i\lambda}}$ in $\mathcal{S}'(\RR)$. Moreover, it is easy to check that $$\lim\limits_{\epsilon\to 0^{+}} {1\over{\epsilon+i\lambda}}={\pi\over 2}\delta -i(vp)\left({1\over\lambda}\right).$$
\par Thus the desired equality follows. For $d^{+}_{\epsilon,l,f}$ we need to change variable according to $\lambda \longleftrightarrow -\lambda$ after considering the Fourier transform of $h$.
\end{proof}

\par We define for $j\in\NN$, $0<j<n-1$, the functions of $r$, with $0\le r <1$, by
$$ w_{j}^{-}(r)=\sum\limits_{k\ge 0} (-1)^{k} \binom{j+k-1}{k} {{r^{2k+n-\alpha}}\over{2k+n-\alpha}},$$	
$$ w_{j}^{+}(r)=\sum\limits_{k\ge 0} (-1)^{k} \binom{j+k-1}{k} {{r^{2k+n+\alpha}}\over{2k+n+\alpha}}.$$
\par We can see, in a complete analogous way as the computations made for $C_{\alpha}$ and $\widetilde{C_{\alpha}}$, that this functions are well defined and that
\begin{equation}\label{c-j} 
\begin{array}{rl}
c_{j}^{-}:=& \lim\limits_{r\to 1^{-}} w_{j}^{-}(r) = {1\over 2} B_{1\over 2} \left( {{n-\alpha}\over{2}}, j-{{n-\alpha}\over{2}} \right), \mbox{ and} \\
c_{j}^{+}:=& \lim\limits_{r\to 1^{-}} w_{j}^{+}(r) = {1\over 2} B_{1\over 2} \left( {{n+\alpha}\over{2}}, j-{{n+\alpha}\over{2}} \right),\end{array} \end{equation}
where $B_{1\over 2}$ is another special function called the \textit{incomplete Beta function}.
\par We now plug all of this definitions and results together to finally obtain an expression for $\Phi_{12}$:
\begin{align*}
& <\Phi_{12},f>  = 
\sum\limits_{\stackrel{l=0}{l \mbox{ \tiny{odd}}}}^{n-2} \sum\limits_{j=1}^{l+1}  {2^{2l-n+j+3}} \binom{n-j-1}{l-j+1}  \left[ \left({1\over i^{n-l-2}} c^{-}_{j} + i^{n-l-2} c^{+}_{j}\right) {\pi\over 2}  \right] \times \\
& \qquad \qquad \times  < \delta , {{\partial^{n-2}}\over{\partial\lambda^{n-l-2}\partial\tau^{l}}} Nf(0,.)> + \\
& \qquad + (-1) \sum\limits_{\stackrel{l=0}{l \mbox{ \tiny{odd}}}}^{n-2} \sum\limits_{j=1}^{l+1}  {2^{2l-n+j+3}} \binom{n-j-1}{l-j+1}  \left({1\over i^{n-l+1}} c^{-}_{j} + i^{n-l+1} c^{+}_{j} \right) \times \\
& \qquad \qquad \times <(vp)\left({1\over\lambda}\right) , {{\partial^{n-2}}\over{\partial\lambda^{n-l-2}\partial\tau^{l}}} Nf(0,.)>.
\end{align*}
\par All we need to do now is to use again lemma 2.8 to get an expression for $\Phi_{2}$. Thus, we have proved the following
\begin{thm} Let $C_{\alpha}$ and $\widetilde{C_{\alpha}}$ be the constans defined as in (\ref{C0}), and let $c^{-}_{j}$ and $c^{+}_{j}$ the constants defined as in (\ref{c-j}). Then 
\begin{align*}
& <\Phi,f> = 4^{n-1}(n-1)!C_{\alpha} \int\limits_{\RR} \int\limits_{\tau>0} {1\over {(\tau-4it)^{{n-\alpha}\over 2}}} {1\over {(\tau+4it)^{{n+\alpha}\over 2}}} \mbox{sgn}(\tau)G_{f}(\tau,t)d\tau dt + \\
& \qquad + 4^{n-1}(n-1)!\widetilde{C_{\alpha}} \int\limits_{\RR} \int\limits_{\tau<0} {1\over {(\tau-4it)^{{n-\alpha}\over 2}}} {1\over {(\tau+4it)^{{n+\alpha}\over 2}}} \mbox{sgn}(\tau)G_{f}(\tau,t)d\tau dt + \\
& \qquad + \sum\limits_{\stackrel{l=0}{l \mbox{ \tiny{odd}}}}^{n-2} \sum\limits_{j=1}^{l+1}  {2^{2l-n+j+3}} \binom{n-j-1}{l-j+1}  \left[ \left({1\over i^{n-l-2}} c^{-}_{j} + i^{n-l-2} c^{+}_{j}\right) {\pi\over 2}  \right] \times \\
& \qquad \qquad \times < \delta , {{\partial^{n-2}}\over{\partial\lambda^{n-l-2}\partial\tau^{l}}} Nf(0,.)> + \\
 & \qquad +(-1) \sum\limits_{\stackrel{l=0}{l \mbox{ \tiny{odd}}}}^{n-2} \sum\limits_{j=1}^{l+1}  {2^{2l-n+j+3}} \binom{n-j-1}{l-j+1}  \left({1\over i^{n-l+1}} c^{-}_{j} + i^{n-l+1} c^{+}_{j} \right) \times \\
 & \qquad \qquad \times <(vp)\left({1\over\lambda}\right) , {{\partial^{n-2}}\over{\partial\lambda^{n-l-2}\partial\tau^{l}}} Nf(0,.)> + \\
 & \qquad + \sum\limits_{k=0}^{n-2} \sum\limits_{l=0}^{n-2} c_{kl} \left[ {1\over{n-2k+\alpha-2}} i^{n-l-2} < {\pi\over 2}\delta+i (vp)\left({1\over\lambda}\right), {{\partial^{n-2}}\over{\partial\lambda^{n-l-2}\partial\tau^{l}}} Nf(0,.)> + \right. \\
&  \qquad \qquad \qquad \left. + {1\over{n-2k-\alpha-2}} {1\over {i^{n-l-2}}} < {\pi\over 2}\delta-i (vp)\left({1\over\lambda}\right), {{\partial^{n-2}}\over{\partial\lambda^{n-l-2}\partial\tau^{l}}} Nf(0,.)> \right],
\end{align*}
where $c_{kl}= \sum\limits_{\stackrel{1\le j\le n-2}{\tiny{j\ge n-k-2}}} 2^{2l-n-j} (-1)^{n-j} \binom{j}{l} \binom{k-l+1}{n-j-2}$.
\end{thm}

\section{A fundamental solution for $L$}

\par Like in the classical case, we have that the distribution $\Phi$ defined in (\ref{fundsol2}) is a well defined tempered distribution and it is a relative fundamental solution for the operator $L$. The proof is identical to the one of theorem 2.1.

\par We will compute the fundamental solution $\Phi$ by means of the Radon transform and the fundamental solution of the operator $L$ in the classical case $\HH_{2n}$. 

\par Let $F\in\mathcal{S}(\RR^{3})$. We assign to $F$ a function $\mathcal{R}F: \RR\times S^{2} \to \RR$ defined by $$\mathcal{R}F(t,\xi)=\int\limits_{\RR^{2}} F(t\xi+u_{1}e_{1}+u_{2}e_{2}) du_{1} du_{2},$$ where $\{\xi,e_{1},e_{2}\}$ is an orthonormal basis of $\RR^{3}$. It is easy to see that this definition does not depend on the choice of the basis. In order to recover $F$ from $\mathcal{R}F$, we consider the space $\mathcal{S}(\RR\times S^{2})$ of the continuous functions $G:\RR\times S^{2}\to \RR$ that are infinitely differentiable in $t$ and satisfy for every $m,k\in\NN_{0}$ that $$\sup\limits_{t\in\RR,\xi\in S^{2}} \left| t^{m} {{\partial^{k}}\over{\partial t^{k}}}G(t,\xi) \right|<\infty.$$ Now for $G\in\mathcal{S}(\RR\times S^{2})$ we define a function $\mathcal{R}^{\ast}G:\RR^{3}\to \RR$ by $$\mathcal{R}^{\ast}G(z)=\int\limits_{S^{2}} G(<z,\xi>,\xi) d\xi.$$ Both assignments are well defined, and moreover $\mathcal{R}:\mathcal{S}(\RR^{3})\to\mathcal{S}(\RR\times S^{2})$ is the \textit{Radon transform}, $\mathcal{R}^{\ast}:\mathcal{S}(\RR\times S^{2})\to\mathcal{S}(\RR^{3})$ is the \textit{dual Radon transform} and they satisfy for every $F\in\mathcal{S}(\RR^{3})$ \begin{equation}\label{inversionradon}-2\pi F = \triangle \mathcal{R}^{\ast} \mathcal{R} F,\end{equation} where $\triangle={{\partial^{2}}\over{\partial z_{1}^{2}}}+{{\partial^{2}}\over{\partial z_{2}^{2}}}+{{\partial^{2}}\over{\partial z_{3}^{2}}}$ is the $\RR^{3}-$Laplacian.

\par Now let us consider the function $\phi$ defined for a fixed $\tau$, $\tau\neq 0$ by $$\phi(\tau,z)={{16n}\over {\pi}} {{4^{2n}(2n-1)!c_{0}}\over{(\tau^{2}+16|z|^{2})^{n+1}}},$$ 
where $c_{0}=-\int\limits_{0}^{1} \sigma^{2n-1}(1+\sigma^{2})^{2n} d\sigma$. The function $\phi(\tau,.)$ is not a Schwartz function on $\RR^{3}$, but we have that $(1+\triangle)^{k}\phi(\tau,.)\in L^{1}(\RR^{3})$, hence $(1+|\xi|^{2})^{k}\widehat{\phi(\tau,.)}(\xi)\in L^{\infty}(\RR^{3})$. With these properties the inversion formula for the Radon transform (\ref{inversionradon}) still holds. The proof follows straightforward from theorem 5.4 of \cite{[S-Sh]}. Let us now compute the Radon transform of the function $\phi$.
\begin{align*}
& \mathcal{R}\phi(\tau,t,\xi)=\int\limits_{\RR^{2}} {{16n}\over {\pi}} {{4^{2n}(2n-1)!c_{0}}\over{(\tau^{2}+16(t^{2}+u_{1}^{2}+u_{2}^{2}))^{n+1}}} du_{1} du_{2} = \\
& \qquad =  {{16n}\over {\pi}} {{4^{2n}(2n-1)!c_{0}}\over{16^{n+1}}}  \int\limits_{\RR^{2}} {1\over{\left({{\tau^{2}}\over 16}+t^{2}+(u_{1}^{2}+u_{2}^{2})\right)^{n+1}}}  du_{1} du_{2} = \\
& \qquad ={{16n}\over {\pi}} {{4^{2n}(2n-1)!c_{0}}\over{16^{n+1}}}  \int\limits_{-\pi\over 2}^{{3\over 2}\pi} \int\limits_{0}^{\infty}  {{\rho}\over{\left({{\tau^{2}}\over 16}+t^{2}+\rho^{2}\right)^{n+1}}}  d\rho d\theta = \\
& \qquad = {{{4^{2n}(2n-1)!c_{0}} }\over (\tau^{2}+16|z|^{2})^{n}},
\end{align*}
where $z=t\xi$. Let $\varphi(\tau,z)={{{4^{2n}(2n-1)!c_{0}} }\over (\tau^{2}+16|z|^{2})^{n}}$.  Now from the expression of the fundamental solution of $L$ in the classical case (see for example 4.3 of \cite{[G-S2]}) we know that
 $$\varphi(\tau,t\xi)=\sum\limits_{k\ge 0} { {(-1)}\over{2k+2n}}  \int\limits_{-\infty}^{\infty} e^{i\lambda t}  L_{k}^{2n-1}\left( {{\lambda}\over 2}|\tau| \right) e^{-{{\lambda}\over 4}|\tau|} |\lambda|^{2n-1} d\lambda.$$

\hfill

\par The first step to compute $\Phi$ is to change to polar coordinates in $\RR^{3}$ the expression given in (\ref{fundsol2}):
\begin{align*}
&<\Phi,f> =  \sum\limits_{k\in\ZZ} \int\limits_{\RR^{3}} {1\over{-|\lambda|(2k+2(p-q))}} <\varphi_{w,k},f> |w|^{2n} dw = \\
& \qquad = \sum\limits_{k\in\ZZ} \int\limits_{S^{2}} \int\limits_{0}^{\infty} {1\over{-|\lambda|(2k+2(p-q))}} <\varphi_{\lambda\xi,k},f> |\lambda|^{2n+2} d\lambda d\xi.
\end{align*}

\par From the absolute convergence of (\ref{fundsol2}) we can interchange the summation symbol with the integral over $S^{2}$. Because of $\triangle e^{i\lambda<\xi,z>}=-|\lambda|^{2}e^{i\lambda<\xi,z>}$, we have that
\begin{align*}
& <\Phi,f> = \int\limits_{S^{2}} \sum\limits_{k\in\ZZ} {(-1)\over{(2k+2(p-q))}} \int\limits_{0}^{\infty} \int\limits_{N(p,q,\HH)} e^{i\lambda<\xi,z>} \theta_{\lambda,k}(\alpha)f(\alpha,z) d\alpha dz |\lambda|^{2n+1} d\lambda d\xi = \\
&\qquad = \int\limits_{S^{2}} \sum\limits_{k\in\ZZ} {1\over{(2k+2(p-q))}} \int\limits_{0}^{\infty} \int\limits_{N(p,q,\HH)} \triangle e^{i\lambda<\xi,z>} \theta_{\lambda,k}(\alpha)f(\alpha,z) d\alpha dz |\lambda|^{2n-1} d\lambda d\xi = \\
& \qquad = \int\limits_{S^{2}} \sum\limits_{k\in\ZZ} {1\over{(2k+2(p-q))}} \int\limits_{0}^{\infty} <\varphi_{\lambda\xi,k}, \triangle f> |\lambda|^{2n-1} d\lambda d\xi.\\
\end{align*}
Next we break the summation indexes according to $k\ge 2q$, $k\le -2p$ and $-2p<k<2q$ to get the splitting $<\Phi,f>=<\Phi_{1},f>+<\Phi_{2},f>$, and as in section we change summation index to get the series starting from $k=0$. From the explicit definition of $\varphi_{\lambda\xi,k}$ we can write \begin{align*}
&<\Phi_{1},f>=\int\limits_{S^{2}} \sum\limits_{k\ge 0} {1\over{2k+2n}} \int\limits_{0}^{\infty} \int\limits_{\RR^{3}} e^{i\lambda<\xi,z>} \times \\
& \qquad \qquad \times <T_{\lambda,k+2q}-T_{\lambda,-k-2p},N\triangle f(.,z)> dz |\lambda|^{2n-1} d\lambda d\xi,
\end{align*}
where $T_{\lambda,k}$ is defined by equations (\ref{Flambdak>0}) and (\ref{Flambdak<0}). By performing similar computations than in section 2 and introducing the function $$G_{f}(\tau,z)=Nf(\tau,z)-\sum\limits_{j=0}^{2n-2} {{\partial^{j}Nf}\over{\partial\tau^{j}}}(0,z){{\tau^{j}}\over{j!}}$$ we get the splitting 
$<\Phi_{1},f>=<\Phi_{11},f>+<\Phi_{12},f>$, where
\begin{align} \label{defphi11}
<\Phi_{11},f> &= \int\limits_{S^{2}} \sum\limits_{k\ge 0} {(-1)\over{2k+2n}}\int\limits_{0}^{\infty} \int\limits_{\RR^{3}} \int\limits_{-\infty}^{\infty} e^{i\lambda<\xi,z>} \times \\ \nonumber
& \qquad \times  \mbox{ sgn}(\tau) L_{k}^{2n-1}\left({2\over{\lambda}}|\tau|\right)e^{-{{\lambda}\over 4}|\tau|} \triangle G_{f}(\tau,z) d\tau dz |\lambda|^{2n-1} d\lambda d\xi,
\end{align}
\begin{align} \label{defphi12}
<\Phi_{12},f> &= 2\int\limits_{S^{2}} \sum\limits_{k\ge 0} {1\over{2k+2n}} \int\limits_{0}^{\infty} \int\limits_{\RR^{3}} e^{i\lambda<\xi,z>} \times \\ \nonumber
& \qquad \times \sum\limits_{\stackrel{l=0}{l \mbox{ \tiny{odd}}}}^{2n-2} \left({2\over{\lambda}}\right)^{l+1} (a_{kl}+b_{kl}) <\delta^{(l)},\triangle Nf(.,z)> dz |\lambda|^{2n-1} d\xi, 
\end{align}
and $a_{kl}$, $b_{kl}$ are the same constant defined in (\ref{akl}) and (\ref{bkl}), respectively. Now let us recall the fact that 
$$\int\limits_{S^{2}} \int\limits_{0}^{\infty} e^{i\lambda<\xi,z>}F(|\lambda|)d\lambda d\xi = {1\over 2} \int\limits_{S^{2}}\int\limits_{-\infty}^{\infty} e^{i\lambda<\xi,z>} F(|\lambda|)d\lambda d\xi,$$ and we apply the dual Radon transform to (\ref{defphi11}).

\par Observe now that $$\int\limits_{-\infty}^{\infty} \int\limits_{\RR^{3}} {{\mbox{sgn}(\tau)G_{f}(\tau,z)}\over{(1+16|z|^{2})^{n+1}}} dz d\tau$$ converges, which we can see by changing to polar variables in $\RR^{3}$ and arguing like in lemma 2.2 of \cite{[G-S2]}.

\par We finally get that 
\begin{align*}
&<\Phi_{11},f> =  {1\over 2} <-2\pi {{16n}\over {\pi}} {{4^{2n}(2n-1)!c_{0}}\over{(\tau^{2}+16|z|^{2})^{n+1}}} , \mbox{ sgn}(\tau)  G_{f}(\tau,z)>= \\
&\qquad = -4^{2n+2} n (2n-1)! c_{0} <{1\over {(\tau^{2}+16|z|^{2})^{n+1}}},\mbox{ sgn}(\tau)G_{f}(\tau,z)>.\end{align*}

\par We have thus proven that the expression defining $\Phi_{11}$ is finite. Then the expression defining $\Phi_{12}$ is also finite, and by Abel's lemma we can write
\begin{align*}
<\Phi_{12},f>&= 2 \lim\limits_{r\to 1^{-}} \lim\limits_{\epsilon\to 0^{+}} \sum\limits_{k\ge 0} \sum\limits_{\stackrel{l=0}{l \mbox{ \tiny{odd}}}}^{2n-2} 2^{l+1} (a_{kl}+b_{kl}) {{r^{2k+2n}}\over{2k+2n}} d_{\epsilon,l,f},
\end{align*}
where
\begin{equation}
\label{delf}
d_{\epsilon,l,f}=\int\limits_{S^{2}} \int\limits_{0}^{\infty}\int\limits_{\RR^{3}} e^{-\epsilon\lambda} e^{i\lambda <\xi,z>} |\lambda|^{2n-l-2} <\delta^{(l)}, \triangle Nf(.,z)> dz d\lambda d\xi.
\end{equation}
We need to compute $\lim\limits_{\epsilon\to 0^{+}} d_{\epsilon,l,f}$. Observing that $\triangle e^{i\lambda <\xi,z>}=-|\lambda|^{2} e^{i\lambda<\xi,z>}$, we have that
\begin{align*}
d_{\epsilon,l,f} &= (-1)^{l+1} \int\limits_{S^{2}}\int\limits_{0}^{\infty}\int\limits_{\RR^{3}} e^{-\epsilon\lambda}e^{i\lambda<\xi,z>} |\lambda|^{2n-l} {{\partial^{l}}\over{\partial\tau^{l}}}Nf(0,z) dz d\lambda d\xi = \\
&= (-1)^{l+1} \int\limits_{\RR^{3}} \int\limits_{\RR^{3}} e^{-\epsilon|x|}|x|^{2n-l-2}e^{i<x,z>} { {\partial^{l}}\over{\partial\tau^{l}}}Nf(0,z) dz dx,
\end{align*}
where we have changed to cartesian coordinates in $\RR^{3}$. To solve this integral let us observe that $$(-1)^{2n-l-2} e^{-\epsilon|x|} |x|^{2n-l-2} = \widehat{{{\partial^{2n-l-2}}\over{\partial\epsilon^{2n-l-2}}}} P_{\epsilon}(x),$$
where $P_{\epsilon}(x)$ is the Poisson kernel and $\hat{ }$ denotes the Fourier transform. Let us write $$d_{\epsilon,l,f}=(-1)^{l}\int\limits_{\RR^{3}} \widehat{{{\partial^{2n-l-2}}\over{\partial\epsilon^{2n-l-2}}}} P_{\epsilon}(x) \left( {{\partial^{l}}\over{\partial\tau^{l}}} Nf(0,.) \right)^{\hat{ }}(x) dx = (-1)^{l} {{\partial^{2n-l-2}}\over{\partial\epsilon^{2n-l-2}}} (P_{\epsilon}\ast h)(0).$$
Taking limit as $\epsilon\to 0^{+}$ we obtain $$\lim\limits_{\epsilon\to 0^{+}} = (-1) (-\triangle)^{{2n-l-2}\over 2} {{\partial^{l}}\over{\partial\tau^{l}}}Nf(0,0),$$ where $(-\triangle)^{{2n-l-2}\over 2}$ is a fractionary exponential of the Laplacian (see for example \cite{[S-Sh]}), which is the operator defined for $g\in\mathcal{S}(\RR^{3})$ by $$(-\triangle)^{{2n-l-2}\over 2} g(x)=\int\limits_{\RR^{3}} |\omega|^{2n-l-2} \widehat{g}(\omega) e^{i<\omega,z>} d\omega.$$
This computation together with proposition 4.8 of \cite{[G-S2]} lets us write
$$<\Phi_{12},f>= \sum\limits_{\stackrel{l=0}{l \mbox{ \tiny{odd}}}}^{2n-2} \sum\limits_{j=1}^{l+1} {1\over{2^{2n-2l-j-3}}} c_{j} \binom{2n-j-1}{l-j-1} (-1) (-\triangle)^{{2n-l-2}\over 2} {{\partial^{l}}\over{\partial\tau^{l}}} Nf(0,0),$$
where each $c_{j}$ is the constant defined in Remark 4.7 of \cite{[G-S2]}.

\par After performing the usual computations for $\Phi_{2}$ we have proved the main theorem of this section:

\begin{thm} Let $c_{0}$ and $c_{j}$ be the constants defined above. Then
\begin{align*}
&<\Phi,f>= -4^{2n+2} n (2n-1)! c_{0} <{1\over {(\tau^{2}+16|z|^{2})^{n+1}}},\mbox{ sgn}(\tau)G_{f}(\tau,z)> + \\
&\qquad +\sum\limits_{\stackrel{l=0}{l\mbox{odd}}}^{2n-2}   \sum\limits_{j=1}^{l+1} {1\over{2^{2n-2l-j-3}}} c_{j} \binom{2n-j-1}{l-j-1} (-1) (-\triangle)^{{2n-l-2}\over 2} {{\partial^{l}}\over{\partial\tau^{l}}}Nf(0,0) + \\
& \qquad + \sum\limits_{-2q<k<0} {1\over{2k+2(p-q)}} \left[ (-1)^{k+1}\sum\limits_{r=0}^{k+2p-1}\binom{k+2p-1}{r}2^{-k-2p+2r+2} \times \right. \\
&\qquad \qquad \left. \times (-1) (-\triangle)^{{2n-r-2}\over 2} {{\partial^{r}}\over{\partial\tau^{r}}}Nf(0,0) + \right. \\
& \qquad \left. + \sum\limits_{l=0}^{2n-2-k-2p} (-1)^{-l-k} \binom{-k-2p+2n-1}{2n-2-l-k-2p} \sum\limits_{r=0}^{l+k-2p} \binom{l+k-2p}{r} \times \right. \\
&\qquad \qquad \left. \times (-1) 2^{-l-k+2p+2r+1}  (-\triangle)^{{2n-r-2}\over 2} {{\partial^{r}}\over{\partial\tau^{r}}}Nf(0,0) \right] + \\
&\qquad + \sum\limits_{0\le k<2q} {1\over{2k+2(p-q)}} \left[  (-1)^{k+1}\sum\limits_{r=0}^{-k+2q-1}\binom{-k+2q-1}{r} \times \right. \\
&\qquad \qquad \left. \times 2^{k-2q+2r+1} (-1)^{r} (-\triangle)^{{2n-r-2}\over 2} {{\partial^{r}}\over{\partial\tau^{r}}}Nf(0,0) + \right. \\
& \qquad \left. + \sum\limits_{l=0}^{2n-2+k-2q} (-1)^{-l+k-2q} \binom{k-2q+2n-1}{2n-2-l+k-2q} \sum\limits_{r=0}^{l-k+2q} \binom{l-k+2q}{r}  \times \right. \\
&\qquad \qquad \left. \times 2^{-l+k-2q+2r+1} (-1)^{r} (-\triangle)^{{2n-r-2}\over 2} {{\partial^{r}}\over{\partial\tau^{r}}}Nf(0,0)\right].\\
\end{align*}
\end{thm}

\hfill

\subsection*{Remark} Let $N$ be a group of Heisenberg type and let $\eta$ be its Lie algebra. So $\eta=V\oplus\mathfrak{z}$, with $\dim V = 2m$ and $\dim \mathfrak{z}=k$. Let $U(V)$ be the unitary group acting on $V$. Then it is known (\cite{[R]}) that $(N\ltimes U(V), U(V))$ is a Gelfand pair, and also in \cite{[R]} were computed the spherical functions. We fix an orthonormal basis of $V$, $\{X_{1},\dots,X_{2m}\}$, and consider the operator $$L=\sum\limits_{j=1}^{2m} X_{j}^{2}.$$ With the same arguments as above, using the Radon transform in $\RR^{k}$ and the fundamental solution of $L$ in the classical Heisenberg group $2m+1$ dimensional, we can recover the fundamental solution of $L$ (see \cite{[K]}, \cite{[R]}).











\end{document}